\documentclass{amsart}

\usepackage[utf8]{inputenc} 
\usepackage[T1]{fontenc} 
\usepackage{amsmath, amsthm, amssymb, booktabs, multirow, graphicx,
  booktabs, bm, float}

\usepackage{mathtools}
\usepackage{enumitem}
\usepackage{dialogue}
\usepackage[stable]{footmisc}
\usepackage[colorlinks]{hyperref}

\usepackage{comment} 

\newtheorem{defin}{Definition}[section]

\newtheorem{theorem}[defin]{Theorem}
\newtheorem{proposition}[defin]{Proposition}
\newtheorem{lemma}[defin]{Lemma}
\newtheorem{corollary}[defin]{Corollary}

\newcommand{\Q}{\mathbb{Q}}
\newcommand{\R}{\mathbb{R}}

\newcommand{\N}{\mathbb{N}}

\newcommand{\sphere}{S}

\newcommand{\sos}[1]{\ensuremath{\Sigma_{#1}}} 
\newcommand{\Ein}{\mathcal{E}_{\mathrm{in}}}
\newcommand{\Eout}{\mathcal{E}_{\mathrm{out}}}

\newcommand{\D}{\mathrm{d}}
\makeatletter
\newcommand{\bigperp}{%
  \mathop{\mathpalette\bigp@rp\relax}%
  \displaylimits
}

\newcommand{\bigp@rp}[2]{%
  \vcenter{
    \m@th\hbox{\scalebox{\ifx#1\displaystyle2.1\else1.5\fi}{$#1\perp$}}
  }%
}
\makeatother

\DeclareMathOperator{\Tr}{Tr}

\DeclareMathOperator{\Harm}{Harm}

\DeclareMathOperator{\poly}{poly}

\DeclarePairedDelimiter{\ordinarySet}{\{}{\}}
\newcommand{\set}[3][]{\ordinarySet[#1]{\,#2 \;#1|\; #3\,}}

\begin{document}

\title{Sums of squares in polynomial time}

\author{Nikolas G\"artner}
\address{N.~G\"artner, Department Mathematik/Informatik, Abteilung
  Mathematik, Universit\"at zu K\"oln, Weyertal~86--90, 50931 K\"oln,
  Germany}
\email{ngaertn2@smail.uni-koeln.de}

\author{Victor Magron}
\address{V.~Magron, Universit\'e de Toulouse; LAAS-CNRS, 
7 avenue du colonel Roche, F-31400 Toulouse, France}
\email{victor.magron@laas.fr}

\author{Frank Vallentin}
\address{F.~Vallentin, Department Mathematik/Informatik, Abteilung
  Mathematik, Universit\"at zu K\"oln, Weyertal~86--90, 50931 K\"oln,
  Germany}
\email{frank.vallentin@uni-koeln.de}

\date{\today}


%
%
%

\begin{abstract}
In this paper, we analyze the bit complexity of deciding whether a given
polynomial can be represented as a sum of squares of polynomials.
We show that the weak membership problem for the sum-of-squares cone lies in $\mathrm{P}$.
Furthermore, we give a polynomial-time algorithm which computes, for a given polynomial and positive parameter $\epsilon$, an $\epsilon$-relaxed closest sum-of-squares polynomial.
\end{abstract}

\maketitle

\markboth{N. G\"artner, V. Magron, and F. Vallentin}{Sum of squares in polynomial time}

\section{Introduction}

Let $p$ be a multivariate polynomial in $n$ variables of even degree $2d$. A fundamental computational task is to determine whether $p$ can be written as a sum of squares of polynomials; that is, to answer the question: 

\smallskip

\textit{Given a polynomial $p$, do there exist polynomials
$q_1, \ldots, q_m$ such that
\[
p = q_1^2 + q_2^ 2 + \cdots + q_m^2?
\]
}

\medskip

The set of polynomials that can be written as a sum of squares forms a convex cone. Thus, the task is equivalent to deciding membership in the sum-of-squares cone. Since this cone is not polyhedral---it has infinitely many extreme rays, each square defining one---it is natural from a computational complexity point of view to consider the weak membership problem, following the classical work of Gr\"otschel, Lov\'asz, and Schrijver \cite{GroetschelLovaszSchrijver1993}. 

To formalize this problem, we first specify how the input polynomial $p$ is encoded. An $n$-variate polynomial $p \in \R[x]$ in the variables $x = (x_1, \ldots, x_n)$ is a finite linear combination of monomials $x^{\alpha} = x_1^{\alpha_1} x_2^{\alpha_2} \cdots x_n^{\alpha_n}$ with real coefficients $p_{\alpha} = p_{\alpha_1, \alpha_2, \ldots, \alpha_n}$,
\[
p(x) = \sum_{\alpha} p_{\alpha} x^{\alpha}
\]
where only finitely many $p_{\alpha}$ are nonzero. If $p \neq 0$, the maximum value of $|\alpha| = \alpha_1 + \alpha_2 + \cdots + \alpha_n$ over all multi-indices $\alpha$ with $p_{\alpha} \neq 0$ is the \emph{degree} of $p$. 

Without loss of generality, we can assume that the input polynomial $p = \sum_{\alpha} p_\alpha x^\alpha$ is homogeneous, that is, $p$ can be written as the sum
$p = \sum_{|\alpha| = {2d}} p_\alpha x^\alpha$, in which all occurring monomials have degree $2d$. Indeed, a polynomial $p$ can be written as a sum of squares if and only if its homogenization
$x^{2d}_{n+1} p(x_1/x_{n+1}, \ldots, x_n/x_{n+1})$ can be written as a sum of squares.  Hence, throughout the paper we assume $n \geq 2$, since deciding whether a homogeneous univariate polynomial is a sum of squares is trivial.

We denote the vector space of homogeneous polynomials of degree $d$ by $\R[x]_d$. This is the vector space of \emph{$d$-forms}. Its dimension equals
\[
N = N(n,d) = \binom{n+d-1}{n-1} = \binom{n+d-1}{d}.
\] 
We represent a $d$-form by its coefficient vector $(p_{\alpha}) \in \R^{N(n,d)}$. 

We define the \emph{sum-of-squares cone}
\[
\Sigma_{n,2d} = \Big\{ (p_\alpha) : p = \sum_{\alpha} p_{\alpha} x^{\alpha} \text{ is a sum of squares} \Big\} \subseteq \R^{N(n,2d)}.
\] 
The \emph{weak membership problem} of $\Sigma_{n,2d}$ is defined as follows: Given a rational vector $(p_\alpha) \in \Q^{N(n,2d)}$ and a rational number $\epsilon > 0$, conclude with one of the following:
\begin{enumerate}
\item ($(p_\alpha)$ is $\epsilon$-close to $\Sigma_{n,2d}$) \hspace*{1ex} $\|(p_\alpha) - (q_\alpha)\| \leq \epsilon$ for some $(q_\alpha) \in \Sigma_{n,2d}$,
\item ($(p_\alpha)$ is not $\epsilon$-deep in $\Sigma_{n,2d}$) \hspace*{1ex} $\|(p_\alpha) - (q_\alpha)\| \leq \epsilon$ for some $(q_\alpha) \not\in \Sigma_{n,2d}$.
\end{enumerate}
Here, $\| \cdot \|$ denotes the standard Euclidean norm on $\R^{N(n,2d)}$.

\medskip

Our first main result is the following:

\begin{theorem}
\label{th:weak-membership}
The weak membership problem for the sum-of-squares cone can be decided in polynomial time.
\end{theorem}

Here, polynomial-time decidability means that there exists a Turing machine that solves the weak membership problem in a number of steps bounded by a polynomial in the bit-length of the input $(p_\alpha) \in \Q^{N(n,2d)}$ and $\epsilon \in \Q$. 

\medskip

However, solving the weak membership problem does not, and cannot, provide a sum-of-squares decomposition if $p$ lies in $\Sigma_{n,2d}$. 

\smallskip

From a practical point of view, the test whether $p$ lies in $\Sigma_{n,2d}$ can be formulated as checking the feasibility of a semidefinite optimization problem using the Gram matrix method, first proposed by Choi, Lam, and Reznick \cite{ChoiLamReznick1995}: A homogeneous polynomial $p \in \R[x_1, \ldots, x_n]$ in $n$ variables of degree $2d$ has a sum-of-squares representation if and only if there is a positive semidefinite matrix $A$ such that
\begin{equation}
\label{eq:Gram}
p(x) = [x]_{d}^{\sf T} A [x]_{d}, \quad \text{with} \quad A \in \mathcal{S}^{N(n,d)}_+,
\end{equation}
where $[x]_{d} \in \R[x_1, \ldots, x_n]^{N(n,d)}$ is the vector of all monomials of degree $d$. In this case, we say that $A$ is a Gram matrix representation of $p$. However, the computational complexity of general semidefinite optimization problems is not yet fully understood. It is unknown whether one can solve---for rational input data---a semidefinite program in polynomial-time in the Turing machine bit model.

\smallskip

Our second main result establishes a polynomial-time algorithm that, for rational polynomials satisfying a normalization condition, computes an approximation of a closest polynomial admitting a rational, positive definite Gram matrix representation.

\begin{theorem}
\label{th:projection-sos}
Let $p$ be a rational, homogeneous polynomial with
$\int_{S^{n-1}} p(x) \, \D \sigma(x) = 1$.
Let $\epsilon > 0$ be a rational parameter. Then one can compute, in polynomial time, a positive definite matrix $A$ so that 
\[
\begin{split}
\|((p - [x]_{d}^{\sf T} A [x]_{d})_\alpha)\| \leq \min\Big\{\|((p-q)_\alpha)\| \; : & \; (q_\alpha) \in \Sigma_{n,2d},\\
& \; \int_{S^{n-1}} q(x)\, \D \sigma(x) = 1\Big\} + \epsilon,
\end{split}
\]
where $\sigma$ is the unique rotationally invariant probability measure on the unit sphere.
\end{theorem}

\subsection{Structure of the paper}

The remainder of the paper is devoted to the proofs of the main theorems. The proof of Theorem~\ref{th:weak-membership} is presented in Section~\ref{sec:weak-membership}, while the proof of Theorem~\ref{th:projection-sos} is given in Section~\ref{sec:projection-sos}.

Our arguments combine several well-known techniques and results from convex optimization, semidefinite programming, and real algebraic geometry.

In Section~\ref{sec:fine-print}, we review the polynomial-time optimization methods based on the ellipsoid method that we shall use throughout the paper. Our presentation follows the now classical monograph of Gr\"otschel, Lov\'asz, and Schrijver~\cite{GroetschelLovaszSchrijver1993}.

Section~\ref{sec:preparation} contains the main preparatory material. Using the theory of spherical harmonics, we show that three natural inner products on $\R[x]_d$ are efficiently equivalent. We also show that all properties of the moment matrix $M$ required later in the paper can be computed in polynomial time. Furthermore, we recall a compactification of the cone of sums of squares due to Blekherman~\cite{Blekherman2004,Blekherman2006} and review the corresponding properties of the associated Loewner--John ellipsoids. Finally, we show that generalized eigenvalue problems arising in our setting can be solved in polynomial time.

\subsection{Related work}

The computational complexity of sums of squares, particularly in connection with hierarchies for polynomial optimization, has been studied by O'Donnell \cite{ODonnell2016} and by Raghavendra and Weitz \cite{Raghavendra2017}. Both works show that, when one simultaneously optimizes several polynomials admitting sum-of-squares representations, the coefficients required in a certificate may have doubly exponential size. Consequently, the bit-size of a standard representation can become exponential. They also identify important classes of problems for which this phenomenon does not occur and for which polynomial-time solvability can be established using the ellipsoid method. This line of research was extended by Gribling, Polak, and Slot \cite{Gribling2023}, who identified algebraic and geometric conditions guaranteeing polynomial-time computability. Related complexity issues for sum-of-squares bounds in copositive programming are discussed by Palomba, Slot, Vargas, and Mastrolilli \cite{Palomba2025}.

\smallskip

Historically, Powers and W\"ormann \cite{PowersWoermann1998} were the first to discuss algorithmic approaches how to find a Gram matrix $Q$.
Peyrl and Parrilo \cite[Proposition~4]{Peyrl2008} proved that the existence of rational $q_1,\ldots,q_m$
satisfying $p=q_1^2+\cdots + q_m^2$ is equivalent to the existence of a rational Gram matrix of $p$.
For polynomials lying in the interior of the cone of sum of squares, Hillar \cite{Hillar2009} showed that if a rational polynomial $p$ admits an invertible Gram matrix, then there is a Gram matrix for $p$ with rational coefficients.

Combining, \cite[Proposition~4]{Peyrl2008} with Theorem~\ref{th:projection-sos} yields the following corollary. Details are worked out in \cite[\S~4.3]{Gaertner2026}.

\begin{corollary}
\label{cor:exact-sos}
Let $\epsilon >0$ be rational and suppose that $p\in \R[x]_{2d}$ is a rational polynomial which admits a Gram matrix $A$ with smallest eigenvalue at least $\epsilon$. Then, one can compute a rational Gram matrix $A'$ of $p$ in polynomial time.
\end{corollary}

In contrast, Scheiderer \cite{Scheiderer2016} proved that there are rational polynomials on the boundary of the cone of sums of squares which do not admit such a representation with rational $q_1,\ldots,q_m$, thereby answering a question raised by Sturmfels.

As an additional by-product, our techniques yield a complexity result about relaxations of polynomial
optimization problems on the sphere $\sphere^{n-1}$. 
Given $p\in \R[x]_{2d}$, let $p_{\min}$ be the minimum of $p$ over $\sphere^{n-1}$, that can be formulated as follows:
%
\begin{align*}
p_{\min}
=
\min_{\mu \in \mathcal{M}_+(\sphere^{n-1})}
\set[\bigg]{
  \int_{\sphere^{n-1}}\, p(x)\, \D \mu(x)
}{
  \int_{\sphere^{n-1}}\, \D \mu(x) = 1
}
,
\end{align*}
where $\mathcal{M}_+(\sphere^{n-1})$ denotes the set of nonnegative Borel measures supported on $\sphere^{n-1}$. 
Instead of optimizing over the full set $\mathcal{M}_+(\sphere^{n-1})$, the framework developed by Lasserre in \cite{Lasserre2011} consists of optimizing over measures of the form $f(x)\, \D \mu(x)$, where $\mu$ is a fixed reference measure with support being exactly $\sphere^{n-1}$ and $f$ is a sum of squares.  
This yields the so-called Lasserre \textit{hierarchy of upper bounds} for $p_{\min}$, indexed by $t \in \N$ with $t \geq d$:
\begin{align}
\label{eq:upperhierarchy}
p_{\min}^{(t)}
=
\min_{f \in \sos{\leq 2t}}
\set[\Big]{
  \int_{\sphere^{n-1}}\, p(x) f(x)\, \D \sigma(x)
}{
  \int_{\sphere^{n-1}}\, f(x)\, \D \sigma(x) = 1
}
,
\end{align}
where $\sos{\leq 2t}$ is the cone of (not necessarily homogeneous) sums of squares of degree at most $2t$.
\begin{corollary}\label{cor:lasserre-hierarchy}
Let $p\in \R[x]_{2d}$ be a rational polynomial.
For every rational $\epsilon > 0$,  
each upper bound $p_{\min}^{(t)}$ of the minimum of $p$ over the unit sphere $\sphere^{n-1}$ can be $\epsilon$-approximated in polynomial time in $N(n,t)$ and in the bit sizes of $\epsilon$ and $(p_\alpha)$. 
\end{corollary}
Details are worked out in \cite[\S~4.4]{Gaertner2026}.

\section{Fine print of convex optimization and computational complexity}
\label{sec:fine-print}

Our results rely on the theory of polynomial-time equivalence between the weak separation problem and the weak optimization problem for convex bodies, which in turn is based on the ellipsoid method. This theory is presented in the book by Gr\"otschel, Lov\'asz, and Schrijver \cite{GroetschelLovaszSchrijver1993}.

Let $K$ be a nonempty convex set\footnote{In the literature, one usually works with non-empty convex sets which are compact. However, for the statement of the problems compactness is not needed.} in the Euclidean space $\R^n$, where $\R^n$ is equipped with the standard inner product $x^{\sf T} y$ and induced norm $\|x\| = \sqrt{x^{\sf T} x}$. We define the \emph{weak optimization problem} as follows: Given $c \in \Q^n$ and a rational number $\epsilon > 0$, find a vector $y \in \Q^n$ such that
\begin{enumerate}
\item ($y$ is $\epsilon$-close to $K$) $\|y - x\| \leq \epsilon$ for some $x \in K$,
\item ($y$ almost minimizes $c^{\sf T} x$ on $K$) $c^{\sf T} y - c^{\sf T} x \leq \epsilon$ for all $x \in K$. 
\end{enumerate}
We define the \emph{weak separation problem}: Given $y \in \Q^n$ and a rational number $\epsilon > 0$, conclude with one of the following:
\begin{enumerate}
\item ($y$ is $\epsilon$-close to $K$) $\|y - x\| \leq \epsilon$ for some $x \in K$,
\item ($y$ can be almost separated from $K$) find a vector $c \in \Q^n$ such that $\|c\| \geq 1$ and for every $x \in K$, $c^{\sf T} x \leq c^{\sf T} y + \epsilon$.
\end{enumerate}
We define the \emph{weak membership problem}: Given $y \in \Q^n$ and a rational number $\epsilon > 0$, conclude with one of the following:
\begin{enumerate}
\item ($y$ is $\epsilon$-close to $K$) $\|y - x\| \leq \epsilon$ for some $x \in K$,
\item ($y$ is not $\epsilon$-deep in $K$) $\|y - x\| \leq \epsilon$ for some $x \not\in K$.
\end{enumerate}

A basic theorem of the theory, which we shall need in this paper, is the following theorem, which can be found in the initial paper by Gr\"otschel, Lov\'asz, and Schrijver \cite{GroetschelLovaszSchrijver1981a}.

\begin{theorem}
\label{th:poly-wopt-iff-poly-wsep}
Let $K$ be a nonempty convex compact set in $\R^n$. Suppose we know a rational point $a_0 \in K$ and positive rational numbers $r,R$ so that 
\begin{equation}\label{inclusions-euclidean-balls}
  B(a_0,r) \subseteq K \subseteq B(a_0,R),
\end{equation}
where $B(a_0,r)$ is the ball of radius $r$ centered at $a_0$. Then there is a polynomial-time algorithm to solve the weak separation problem if and only if there is a polynomial-time algorithm to solve the weak optimization problem. Here, the polynomial is in $n$, $\log_2 \frac{R}{r}$, $\log_2(1/\epsilon)$, and the bit size of the data $a_0$, and $c \in \Q^n$, respectively, $y \in \Q^n$.
\end{theorem}

It is known (see \cite[Figure 2.2]{GroetschelLovaszSchrijver1993}) that if the weak separation problem can be solved in polynomial time, then also the weak membership problem can be solved in polynomial time.

For semidefinite programs, that is, convex optimization problems of the form
\begin{equation}
\label{eq:SDP}
\mathrm{val} = \inf\{\langle C, X \rangle : X \in \mathcal{S}^n_+, \; \langle A_j, X \rangle = b_j \; (j = 1, \ldots, m)\},
\end{equation}
where symmetric matrices $C$, $A_1, \ldots, A_m \in \mathcal{S}^n$, and $b_1, \ldots, b_m \in \R$ are given as input, and the positive semidefinite matrix $X \in \mathcal{S}^n_+$ is the optimization variable, and where $\langle X, Y \rangle = \Tr(XY)$ denotes the trace inner product for symmetric matrices, the above theorem implies:

\begin{theorem}
\label{th:sdppolynomialtime}
Consider the semidefinite program~\eqref{eq:SDP}
with rational input data $C$, $A_1, \ldots, A_m$, and $b_1, \ldots, b_m$, which define the set of feasible solutions
\[
\mathcal{F} = \{ X \in \mathcal{S}^n_+ : \langle A_j, X \rangle = b_j \; (j = 1, \ldots, m) \}.
\]
Suppose that we know a rational point
$X_0 \in \mathcal{F}$ and positive rational numbers $r$, $R$ so that
\[
X_0+B(r) \subseteq \mathcal{F} \subseteq X_0 + B(R),
\]
where $B(r)$ is the ball of radius $r$, centered at the zero matrix, in the $d$-dimensional subspace
\[
L = \{X \in \mathcal{S}^n : \langle A_j, X \rangle= 0 \text{ for } j = 1,
\ldots, m\}.
\]
For every positive rational number $\epsilon>0$ one can compute, in
polynomial time, a rational matrix $X^* \in \mathcal{F}$ such that
\[
\langle C, X^* \rangle - \mathrm{val} \leq \epsilon,
\]
where the running time is bounded by a polynomial in $n$, $m$, $\log_2 \frac{R}{r}$,
$\log_2(1/\epsilon)$, and the bit size of the data $X_0$, $C$,
$A_1, \ldots, A_m$, and $b_1, \ldots, b_m$.
\end{theorem}

Here, we use the Frobenius norm $\|X\|_F = \langle X, X \rangle^{1/2}$ for defining the balls $B(r)$ and $B(R)$.

In fact, de Klerk and Vallentin~\cite{deKlerkVallentin2016} showed that one can also prove this result using a short-step, primal interior-point method, rather than using the more theoretical ellipsoid method.

\section{Preparation}
\label{sec:preparation}

\subsection{Efficiently equivalent norms and spherical harmonics}
\label{ssec:norms}

The space $\R[x]_d$ can be endowed with a Euclidean structure by introducing an inner product. For our purposes, three different inner products will be relevant. Let $p = \sum_{|\alpha| = d} p_{\alpha} x^{\alpha}$ and $q = \sum_{|\alpha| = d} q_{\alpha} x^{\alpha}$ be homogeneous polynomials of degree $d$.

The \emph{$L^2$-inner product} is defined by
\[
(p,q)_{L^2} = \int_{S^{n-1}} p(x) q(x) \, \D \sigma(x),
\]
where $\sigma$ denotes the unique rotationally invariant probability measure on the unit sphere.

The \emph{apolar inner product}\footnote{In the literature, the apolar inner product is also referred to as the Fischer, Calder\'on, Bombieri, or Bombieri-Weyl inner product} is
\[
(p,q)_a = \frac{1}{d!} \sum_{\alpha} \alpha! p_{\alpha} q_{\alpha},
\]
where $\alpha! = \alpha_1! \alpha_2! \cdots \alpha_n!$.

Finally, the \emph{standard inner product} is
\[
(p,q) = \sum_{\alpha} p_{\alpha} q_{\alpha}.
\]

The associated induced norms are denoted by $\| \cdot \|_{L^2}$, $\| \cdot \|_{a}$, and $\| \cdot \|$, respectively. In addition, we also consider the $L^p$-norm defined by
\[
\|q\|_{L^p} = \Big(\int_{S^{n-1}} |q(x)|^p \, \D \sigma(x)\Big)^{1/p}.
\]

Since $\mathbb{R}[x]_d$ is finite-dimensional, all these norms are equivalent. However, our analysis requires quite explicit bounds on the equivalence constants. We say that two norms $\|\cdot\|$ and $\|\cdot\|'$ on $\R[x]_d$ are \emph{efficiently equivalent}
if one can compute rational constants $k, K$ in time $O(\poly(N))$ such that for every $p \in \R[x]_d$,
\[
k \|p\|' \leq \|p\| \leq K \|p\|'
\]
holds. In particular, then the bit sizes of $k$ and $K$ are bounded by $O(\poly(N))$.\footnote{In the cases considered in this paper, one can even provide analytic expressions for $k$ and $K$. Details can be found in \cite{Gaertner2026}.}

\smallskip

Now we will show that all three norms are efficiently equivalent.

\subsubsection{Standard inner product vs.\ apolar inner product}

The monomial basis of $\mathbb{R}[x]_d$ is orthogonal with respect to both the standard inner product and the apolar inner product. Consequently, we obtain the following estimate.

\begin{lemma}
\label{lem:standard-vs-apolar}
For every polynomial $p \in \R[x]_d$, one has
\[
\frac{\min\{\alpha! : |\alpha| = d\}}{d!} (p,p) \leq (p,p)_a \leq (p,p).
\]
In particular, the standard norm $\| \cdot \|$ and the apolar norm $\| \cdot \|_a$ are efficiently equivalent.
\end{lemma}

\subsubsection{$L^2$- vs.\ apolar inner product}

The $L^2$-inner product and the apolar inner product are both invariant under the action of the orthogonal group $O(n)$; see for example the exposition \cite{CoifmanWeiss1968} by Coifman and Weiss. On the other hand, the standard inner product is not $O(n)$-invariant.

To understand the exact relationship between the $L^2$- and the apolar inner product, it is therefore natural to decompose $\R[x]_d$ into $O(n)$-irreducible subspaces. It is well known (cf.\ \cite[Theorem (2.12)]{CoifmanWeiss1968}) that this decomposition is given by
\begin{equation}
\label{eq:harmonic-decomposition}
\R[x]_d = \Harm_d \; \oplus \; \omega \Harm_{d-2} \; \oplus \; \omega^2 \Harm_{d-4} \; \oplus \; \cdots,
\end{equation}
where $\textrm{Harm}_k$ is the space of \textit{harmonic polynomials} which are homogeneous and of degree $k$, and where 
\begin{equation}
\label{eq:omega-def}
\omega = x_1^2 + \cdots + x_n^2.
\end{equation}
Recall that a polynomial $p$ is called \textit{harmonic}, if it vanishes under the Laplace operator:
\begin{equation*}
  \Delta p = \frac{\partial^2 p}{\partial x_1^2} + \cdots + \frac{\partial^2 p}{\partial x_n^2} = 0.
\end{equation*}
The dimension of $\Harm_k$ equals
\[
h_k = N(n,k) - N(n,k-2) = \binom{n+k-1}{n-1} - \binom{n+k-3}{n-1}.
\]
Because $\Harm_k$ is $O(n)$-irreducible, it carries an $O(n)$-invariant inner product, which is unique up to scaling. These scaling factors, one for every degree $k$, make the difference between the $L^2$- and the apolar inner product. The scaling factor has been worked out by Coifman and Weiss \cite[Theorem (3.15)]{CoifmanWeiss1968}\footnote{In fact, in \cite{CoifmanWeiss1968} this is stated only for $n \geq 3$ and for the action of the special orthogonal group. However, the same argument works also for $n = 2$ when using the noncommutative orthogonal group $O(2)$.}: Let $p, q \in \Harm_k$, then
\[
(p,q)_{L^2} = A_k (p,q)_a, \text{ with } A_k = a_k^2 / h_k,
\]
and
\[
a_k^2 = \left( 1 + \sum_{j = 1}^{\lfloor k/2 \rfloor}
\frac{\beta_1 \beta_2 \cdots \beta_j}{\alpha_1 \alpha_2 \cdots \alpha_j}
\right)^{-1},
\]
with
\[
\alpha_j = 2j(n+2j-3) \; \text{ and } \; \beta_j = (k-2j+1)(k-2j+2).
\]
Using the direct sum decomposition~\eqref{eq:harmonic-decomposition} which is orthogonal for the $L^2$- and for the apolar inner product, we can compute the equivalence constants. We need to be a bit careful here since the apolar inner product is defined for every degree separately. For this note that for any harmonic polynomial of degree $\ell$ we have
\[
(\omega p, \omega p)_a = (p, \Delta \omega p)_a,
\]
and
\[
\Delta \omega p = (4\ell + 2n) p,
\]
so that
\[
(\omega p, \omega p)_a = (4\ell + 2n) (p,p)_a.
\]
By induction, we can compute $(\omega^{r+1} p, \omega^{r+1} p)_a$ in a similar way, by using the identity
\[\Delta \omega^{r+1} p = (r+1)(4r + 4\ell + 2n) \omega^r p,
\]
see for example \cite[Lemma 3.5.3]{Simon2015a}\footnote{The factor $2$ in (3.5.11) is wrong in \cite{Simon2015a}; it should be $1$.}.

\smallskip

Because the $A_k$'s and the coefficients arising in inner products of the form
$(\omega^{r+1} p, \omega^{r+1} p)_a$ 
can be computed in time $O(\poly(N))$, we can state the following lemma.

\begin{lemma}
\label{lem:l2-vs-apolar}
The $L^2$-norm  $\| \cdot \|_{L^2}$ and the apolar norm $\| \cdot \|_a$ are efficiently equivalent.
\end{lemma}

\subsubsection{$L^2$- vs.\ standard inner product}

From Lemma~\ref{lem:standard-vs-apolar}, Lemma~\ref{lem:l2-vs-apolar}, and transitivity of efficient equivalence, we immediately derive the following lemma.

\begin{lemma}
\label{lem:l2-vs-standard}
The $L^2$-norm $\| \cdot \|_{L^2}$ and the standard norm $\| \cdot \|$ are efficiently equivalent.
\end{lemma}

\subsection{The moment matrix}

We consider the \textit{moment matrix} $M$ given by entries
\begin{equation}
\label{eq:def-moment-matrix}
M_{\alpha, \beta} = \int_{S^{n-1}} x^{\alpha + \beta} \, \D \sigma(x)
\end{equation}
for multi-indices $\alpha, \beta$ with $|\alpha| = |\beta| = d$.

We need to determine the entries $M_{\alpha, \beta}$ explicitly. Folland~\cite{Folland2001} computed the moments for the Lebesgue measure $\mu$, giving
\[
\int_{S^{n-1}} x_1^{\alpha_1} \cdots x_n^{\alpha_n} \, \D \mu(x) =
    \begin{cases}
    0, &\text{if some} \ \alpha_j \ \text{is odd},\\
    \dfrac{2 \, \Gamma(\beta_1) \Gamma(\beta_2) \cdots \Gamma(\beta_n)}{\Gamma(\beta_1 + \beta_2 + \cdots + \beta_n)}, & \text{if all} \ \alpha_j \ \text{are even,}
    \end{cases}
\]
with $\beta_j = \frac{1}{2}(\alpha_j +1)$. Normalizing yields
\[
\int_{S^{n-1}} x_1^{\alpha_1} \cdots x_n^{\alpha_n} \, \D \sigma(x) =
\begin{cases}
0, & \text{if some} \ \alpha_j \ \text{is odd},\\
\dfrac{\prod\limits_{j=1}^{n} \prod\limits_{k=0}^{\alpha_j/2-1}\left(\alpha_j - 1 - 2k\right)}{ \prod\limits_{k=1}^{\left|\alpha\right|/2}\left(n+\left|\alpha\right| -2k\right)}, & \text{if all} \ \alpha_j \ \text{are even.}
\end{cases}
\]
The last formula shows that one can compute the (rational) entries of $M$ in time $O(\poly(N))$.

The moment matrix $M$ is positive definite. In fact, one can use the considerations in Section~\ref{ssec:norms} to efficiently determine a lower bound for the smallest eigenvalue of $M$ and an upper bound for the largest eigenvalue of $M$ using the Rayleigh-Ritz-principle. For any polynomial $p \in \R[x]_d$, given by $\sum_{|\alpha| = d} p_\alpha x^{\alpha}$, we have 
\begin{equation}
\label{eq:moment-matrix-L2-norm}
\begin{split}
(p_\alpha)^{\sf T}  M (p_\alpha)
& \; = \; \int_{S^{n-1}} \sum_{|\alpha|, |\beta| = d} p_\alpha p_\beta x^{\alpha + \beta} \, \D \sigma(x)\\
& \; = \; \int_{S^{n-1}} \left(\sum_{|\alpha| = d} p_\alpha x^{\alpha}\right)^2 \, \D \sigma(x) \\
& \; = \; \left\|p\right\|_{L^2}^2.
\end{split}
\end{equation}
Therefore,
\begin{equation*}
\frac{(p_\alpha)^{\sf T} M (p_\alpha)}{(p_\alpha)^{\sf T} (p_\alpha)}
= \frac{\left\|p\right\|_{L^2}^2}{\left\|p\right\|^2},
\end{equation*}
and by the Rayleigh-Ritz principle, minimizing (resp.\ maximizing) this ratio over $\R[x]_d \setminus \{0\}$ provides the smallest (resp.\ largest) eigenvalue of $M$. Now, Lemma~\ref{lem:l2-vs-standard} can be used to derive the promised bounds.

\subsection{Cones of polynomials and Loewner-John ellipsoids}

Inside $\R[x]_{2d}$ we consider three convex cones. The \emph{cone of sum of squares} is defined by
\[
\sos{n,2d} = \{ p \in \R[x]_{2d} : p = q_1^2 + \cdots + q_m^2 \text{ for some } q_1, \ldots, q_m \in \R[x]_d \}.\footnote{Throughout this paper, we implicitly identify a cone of polynomials with the corresponding cone of coefficient vectors with respect to the monomial basis.}
\]
The \emph{cone of nonnegative forms} is
\[
\mathcal{P}_{n,2d} = \{ p \in \R[x]_{2d} : p(x) \geq 0 \text{ for all } x \in \R^n\}.
\]
With respect to the apolar inner product, the dual cone of $\mathcal{P}_{n,2d}$ is given by
\[
\mathcal{P}^*_{n,2d} = \{ p \in \R[x]_{2d} : (p,q)_a \geq 0 \text{ for all } q \in \mathcal{P}_{n,2d}\}.
\]
This follows from the fact that one can evaluate a polynomial $p$ at a point $v = (v_1, \ldots, v_n) \in \R^n$ via the apolar inner product 
\[
((v \cdot x)^{2d}, p)_a =
\frac{1}{(2d)!} \sum_{|\alpha| = 2d} \alpha! \frac{(2d)!}{\alpha!} v^{\alpha} p_{\alpha} = p(v),
\]
where we applied the multinomial theorem:
\[
(v \cdot x)^{2d} = \Big(\sum_{i=1}^n v_i x_i\Big)^{2d} = 
\sum_{|\alpha| = 2d} \frac{(2d)!}{\alpha!} v^{\alpha} x^{\alpha}.
\]
Consequently, the dual cone is the \emph{cone of $2d$-th powers of linear forms}
\[
\mathcal{P}^*_{n,2d} = \{ p \in \R[x]_{2d} : 
p = q_1^{2d} + \cdots + q_m^{2d} \text{ for some } q_1, \ldots, q_m \in \R[x]_1\}
\]

The three cones satisfy the inclusions
\begin{equation}\label{inclusions-sos-not-normalized}
  \mathcal{P}^*_{n,2d} \; \subseteq \;
  \sos{n,2d} \; \subseteq \;
  \mathcal{P}_{n,2d}.
\end{equation}

In \cite{Blekherman2004}, Blekherman studies the geometry of these three cones. In particular, he computes Loewner-John ellipsoids of appropriate compact sections. 

We intersect $\R[x]_{2d}$ with the affine hyperplane
\[
L_{n,2d} = \Big\{p \in \R[x]_{2d} : \int_{S^{n-1}} p(x)\, \D \sigma(x) = 1\Big\},
\]
which yields the three convex bodies
\[
\overline{\mathcal{P}^*_{n,2d}} \; = \; \mathcal{P}^*_{n,2d} \cap L_{n,2d}, \;\;
\overline{\sos{n,2d}} \; = \; \sos{n,2d} \cap L_{n,2d}, \;\;
\overline{\mathcal{P}_{n,2d}} \; = \; \mathcal{P}_{n,2d} \cap L_{n,2d}.
\]

In general, it is well-known that every convex body $C$ in a finite-dimensional Euclidean space $E$ is contained in a unique ellipsoid of minimal volume $\Ein(C)$ and contains a unique ellipsoid of maximal volume $\Eout(C)$, see \cite[Section 2.2]{Blekherman2004}. These ellipsoids are called Loewner-John ellipsoids. If $C$ is centrally symmetric, that is $C = -C$, then 
\[
\Ein(C) \subseteq C \subseteq \sqrt{\dim E} \; \Ein(C),
\]
and similarly,
\[
\frac{1}{\sqrt{\dim E}} \; \Eout(C) \subseteq C \subseteq \Eout(C).
\]

The \emph{coefficient of symmetry} of a compact convex body $C$ with respect to a point $v$ in the interior of $C$ is defined as
\begin{equation*}
  \max \set{\alpha \in \R}{-\alpha(C-v) \, \subseteq \, C-v}.
\end{equation*}

In \cite{Blekherman2004} Blekherman determines the coefficient of symmetry of $\overline{\mathcal{P}_{n,2d}^*}$ as well as its Loewner-John ellipsoid $\Eout(\overline{\mathcal{P}_{n,2d}^*})$. This implies, \cite[Theorem 7.6]{Blekherman2004}:

\begin{proposition}
\label{prop:blekherman-inner-ball}
$\overline{\mathcal{P}_{n,2d}^*}$ contains the ball of radius
\begin{equation}
\label{eq:blekherman-inner-radius}
\frac{d! \, \Gamma\left(\frac{2d+n}{2}\right)}{\Gamma\left( \frac{4d+n}{2}\right) \sqrt{N(n,d) - 1}}
\end{equation}
with respect to the $L^2$-norm $\|\cdot\|_{L^2}$, centered at $\omega^d$, as defined in~\eqref{eq:omega-def}.
\end{proposition}

Blekherman \cite{Blekherman2004} 
computed the coefficient of symmetry of $\overline{\mathcal{P}_{n,2d}}$ as well as its Loewner-John ellipsoid $\Ein(\overline{\mathcal{P}_{n,2d}})$, implying \cite[Corollary 6.7]{Blekherman2004}:

\begin{proposition}
\label{prop:blekherman-outer-ball}
$\overline{\mathcal{P}_{n,2d}}$ is contained in the ball of radius
\begin{equation*}
\sqrt{N(n,d) - 1} = \sqrt{\binom{n+d-1}{d} - 1}
\end{equation*}
with respect to the $L^2$-norm $\| \cdot \|_{L^2}$, centered at the polynomial $\omega^{d}$.
\end{proposition}

Both bounds for the above radii depend on $n$ as well as $d$.
Their bit sizes are bounded by $O(\poly(N))$. Since 
\[
\overline{\mathcal{P}^*_{n,2d}} \; \subseteq \; \overline{\sos{n,2d}} \; \subseteq \; \overline{\mathcal{P}_{n,2d}},
\]
these results provide suitable balls for our purpose.

Blekherman further noted in \cite[Proof of Theorem 4.1]{Blekherman2006} that one can find a ball containing $\overline{\sos{n,2d}}$ of radius $2 \cdot 4^d$, thus depending only on $d$, by applying the reverse H\"older inequality due to Duoandikoetxea \cite{Duoandikoetxea1987}. We recall Duoandikoetxea's result here, as we shall use it later.

\begin{proposition}
\label{prop:reverse-holder-ineq}
Let $P$ be any polynomial of degree $k$. Then, if $2<p<\infty$,
\begin{equation*}
\left\|P\right\|_{L^p}
\leq p^{k/2}\, \left\|P\right\|_{L^2}
\end{equation*}
and if $0<p<2$,
\begin{equation*}
\left\|P\right\|_{L^2}
\leq 4^{k(2/p-1)}\, \left\|P\right\|_{L^p}.
\end{equation*}
\end{proposition}

\subsection{Polynomial-time solvability of the generalized eigenvalue problem}

The generalized eigenvalue problem for two symmetric matrices $B$ and $C \in \mathcal{S}^n$ amounts to finding the largest $\lambda \in \R$ such that there is a unit vector $x \in S^{n-1}$ so that $Cx = \lambda Bx$ holds. If $B$ is the identity matrix, then $\lambda$ is the largest eigenvalue of $C$, denoted by $\lambda_{\max}(C)$ (respectively, we denote by $\lambda_{\min}(C)$ the smallest eigenvalue of $C$). We consider the more general case when $B$ is a positive definite matrix. More precisely, we show the following theorem.

\begin{theorem}
\label{th:generalized-eigenvalue-sdp}
Let $B \in \mathcal{S}^n_{++}$ be a positive definite, rational matrix and let $C \in \mathcal{S}^n$ be a rational matrix. Let $\epsilon > 0$ be a rational number. Let $\lambda$ be the generalized eigenvalue of $B$ and $C$. Then, one can determine a rational number $\mu$ with $\lambda - \epsilon \leq \mu \leq \lambda$ in polynomial time by solving the semidefinite program
\begin{equation}
\label{eq:generalized-eigenvalue-sdp}
\max\{\langle C, X \rangle : X \in \mathcal{S}^n_+, \; \langle B, X \rangle = 1\}.
\end{equation}
The running time is polynomial in the bit size of the input data $n$, $C$, $B$, $\log 1/\epsilon$.
\end{theorem}

Since we are not aware of a suitable reference for this theorem, we give a proof here. We'll apply the fine print theorem for the polynomial-time solvability of semidefinite optimization, Theorem~\ref{th:sdppolynomialtime}. For this, the following lemma is useful.

\begin{lemma}
\label{lem:generalized-eigenvalue-sdp}
Consider the set of feasible solutions of~\eqref{eq:generalized-eigenvalue-sdp}
\[
\mathcal{F} = \{X : X \in \mathcal{S}^n_+, \; \langle B, X \rangle = 1\}.
\]
Then, for $X_0 = \frac{1}{n} B^{-1}$, the ball $X_0+B(r)$ with $r = \frac{1}{n \lambda_{\max}(B)}$ is contained in $\mathcal{F}$. The ball $X_0+B(R)$ with $R = \frac{2\sqrt{n}}{\lambda_{\min}(B)}$ contains $\mathcal{F}$. Here, balls are defined with respect to the Frobenius norm of a symmetric matrix $\|A\|_F = \sqrt{\Tr(A^2)}$.
\end{lemma}

\begin{proof}
For proving the statement about the inner ball, let $X = X_0 + A$ be such that $\langle B, A \rangle = 0$ and $\|A\|_F \leq r$. Then, $X$ is positive semidefinite because if we split $X$ into the sum
\[
X = (X_0 - r I) + (r I + A),
\]
then both summands are positive semidefinite. Indeed, we have for the first summand
\[
\lambda_{\min}(X_0 - rI) = \lambda_{\min}\left(\frac{1}{n}B^{-1}\right) - r = \frac{1}{n \lambda_{\max}(B)} - r = 0,
\]
and for the second summand
\[
\lambda_{\min}(rI + A) = r + \lambda_{\min}(A) \geq 0,
\]
because
\[
r \geq \|A\|_F \geq -\lambda_{\min}(A).
\]

\smallskip

For proving the statement about the outer ball, let $X = X_0 + A \in \mathcal{F}$ be feasible. Then, $X \in \mathcal{S}^n_+$, and so
\[
1 = \langle B, X \rangle \geq \langle \lambda_{\min}(B) I, X \rangle = \lambda_{\min}(B) \Tr(X) \geq \lambda_{\min}(B) \lambda_{\max}(X).
\]
Furthermore,
\[
\|X\|_F \leq \sqrt{n} \lambda_{\max}(X),
\]
and, hence,
\[
\|X\|_F \leq \frac{\sqrt{n}}{\lambda_{\min}(B)}.
\]
Thus,
\[
\|A\|_F = \|X_0 + A - X_0\|_F
\leq \|X\|_F + \|X_0\|_F
\leq \frac{2\sqrt{n}}{\lambda_{\min}(B)} = R. \qedhere
\]
\end{proof}

\begin{proof}[Proof of Theorem~\ref{th:generalized-eigenvalue-sdp}]
We first determine $\lambda_{\max}(B)$ and $\lambda_{\min}(B)$. That is, we solve the semidefinite program~\eqref{eq:generalized-eigenvalue-sdp} with $C \leftarrow B$ and $B \leftarrow I$ (respectively with $C \leftarrow B^{-1}$) to determine a rational approximation of $\lambda_{\max}(B)$ (respectively for $\lambda_{\min}(B)$). The polynomial-time solvability of these two semidefinite programs is guaranteed by Theorem~\ref{th:sdppolynomialtime} and Lemma~\ref{lem:generalized-eigenvalue-sdp}. Note that the encoding length of the rational approximations of 
$\lambda_{\max}(B)$ and $\lambda_{\min}(B)$ is again polynomial in the input data. Consequently, these approximations can be used in the last step, when we apply Theorem~\ref{th:sdppolynomialtime} and Lemma~\ref{lem:generalized-eigenvalue-sdp} to solve~\eqref{eq:generalized-eigenvalue-sdp} in polynomial time.
\end{proof}

\section{Proof of Theorem~\ref{th:weak-membership}}
\label{sec:weak-membership}

As discussed in Section~\ref{sec:preparation}, we may assume that the input polynomial is homogeneous. We also normalize  $p$ so that it  lies in $\overline{\Sigma_{n,2d}}$. This scaling can be computed in polynomial time, because it amounts to computing
\[
\int_{S^{n-1}} p(x) \, \D \sigma(x) = \mathbf{e}^{\sf T} M (p_{\alpha}),
\]
where $\mathbf{e}$ is the all-ones vector, $M$ is the moment matrix~\eqref{eq:def-moment-matrix} and $(p_{\alpha})$ denotes the coefficient vector of $p$.

We want to apply Theorem~\ref{th:poly-wopt-iff-poly-wsep} to $K = \overline{\sos{n,2d}}$. For this we need to exhibit a polynomial $a_0 \in \Sigma_{n,2d}$ and radii $r, R$. We take $a_0 = \omega^d$, which, by using the multinomial theorem, in the monomial basis has the following expansion
\[
\omega^d = (x_1^2 + \cdots + x_n^2)^d = \sum_{|\alpha| = d} \frac{d!}{\alpha!} x^{2\alpha}.
\]
Suitable radii $r$ and $R$ are provided by Blekherman's computation of the Loewner-John ellipsoids. The radii of the corresponding balls in Proposition~\ref{prop:blekherman-inner-ball} and Proposition~\ref{prop:blekherman-outer-ball} are given in the $L^2$-norm, but through Lemma~\ref{lem:l2-vs-standard} one can translate these bounds efficiently to the standard norm.

By Theorem~\ref{th:poly-wopt-iff-poly-wsep} it now suffices to show that one can solve the following weak optimization problem in polynomial time
\begin{equation}
\label{eq:weak-opt-sos}
\max\{(f_{\alpha})^{\sf T} (p_{\alpha}) : (p_{\alpha}) \in \overline{\sos{n,2d}}\},
\end{equation}
where $f = \sum_{|\alpha| = 2d} f_{\alpha} x^\alpha \in \Q[x]_{2d}$.
This is a convex optimization problem with a linear objective. Therefore, maximizers are extreme points of $\overline{\sos{n,2d}}$. Extreme points are precisely the normalized squares, $p = q^2$ for $q \in \R[x]_d$. Thus, \eqref{eq:weak-opt-sos} simplifies to
\[
\max\left\{(f_{\alpha})^{\sf T} ((q^2)_{\alpha}) : q \in \R[x]_d, \; \int_{S^{n-1}} q^2(x) \, \D \sigma(x) = 1\right\},
\]
which is equivalent to the generalized eigenvalue problem
\[
\max\{ (q_\alpha)^{\sf T} C(f) (q_{\alpha}) : q \in \R[x]_d, \; (q_\alpha)^{\sf T} M (q_\alpha) = 1\},
\]
with matrix $C(f)_{\alpha, \beta} = f_{\alpha + \beta}$ and moment matrix $M$ defined in~\eqref{eq:def-moment-matrix}.

Theorem~\ref{th:generalized-eigenvalue-sdp} therefore yields a polynomial-time algorithm for solving this optimization problem, completing the proof of Theorem~\ref{th:weak-membership}. \qed

\section{Proof of Theorem~\ref{th:projection-sos}}
\label{sec:projection-sos}

For the proof we will use the Gram matrix method. Given a symmetric matrix $A \in \mathcal{S}^{N(n,d)}$, indexed by $(\alpha,\beta)$ satisfying $\lvert \alpha \rvert, \lvert\beta\rvert=d$,
consider the polynomial
\begin{equation*}
p_A(x) = [x]_d^{\sf T} \, A \, [x]_d = 
  \sum\limits_{\left|\gamma\right|=2d} 
  \biggl(\;  \sum\limits_{\alpha + \beta = \gamma} A_{\alpha,\beta}\biggr)
  \, x^{\gamma},
\end{equation*}
where $[x]_d$ is the vector of all monomials of degree $d$. A homogeneous polynomial $p$ of degree $2d$ can be written as a sum of squares if and only if there exists a positive semidefinite matrix $A$ so that $p = p_A$. Thus, this test amounts to solving the feasibility problem of a semidefinite program. 

The following lemma relates the $L^2$-norm of $p_A$ to the Frobenius norm of the matrix $A$.

\begin{lemma}
\label{lem:relation-L2-Frobenius-norm}
Given $A \in \mathcal{S}^{N(n,d)}$ and $p_A = [x]_d^{\sf T} \, A \, [x]_d$, then
\begin{equation}\label{ineq:L2-leq-Frobenius}
\left\|p_A\right\|_{L^2} \leq {4^{2d} \, N(n,d) \, \lambda_{\max}(M)} \left\|A\right\|_F,
\end{equation}
where $M$ is the moment matrix~\eqref{eq:def-moment-matrix}. Furthermore, if $A$ is positive semidefinite, then
\begin{equation}
\label{ineq:Frobenius-leq-L2}
\lambda_{\min}(M) \, \left\|A\right\|_F \leq \left\|p_A\right\|_{L^2}.
\end{equation}
\end{lemma}

\begin{proof}
To prove \eqref{ineq:L2-leq-Frobenius}, observe that
\begin{equation*}
\left\|p\right\|_{L^2}^2
= \int_{S^{n-1}} ([x]_d^{\sf T} \, A \, [x]_d)^2 \, \D \sigma(x)
= \int_{S^{n-1}} \langle A, [x]_d [x]_d^{\sf T} \rangle^2 \, \D \sigma(x).
\end{equation*}
By the Cauchy-Schwarz inequality for the trace inner product,
\begin{equation*}
\int_{S^{n-1}}  \langle A, [x]_d [x]_d^{\sf T} \rangle^2 \, \D \sigma(x)
\leq \left\|A\right\|_F^2 \int_{S^{n-1}} \left\|[x]_d [x]_d^{\sf T}\right\|_F^2 \, \D \sigma(x).
\end{equation*}
Since
\begin{equation*}
\left\|[x]_d [x]_d^{\sf T}\right\|_F^2 
= \biggl( \sum\limits_{\lvert \alpha \rvert=d} x^{2\alpha}  \biggr)^2,
\end{equation*}
we see that the integral $\int_{S^{n-1}} \left\|[x]_d [x]_d^{\sf T}\right\|_F^2 \, \D \sigma(x)$ equals the squared $L^2$-norm of the polynomial
$\sum_{\lvert \alpha\rvert=d} x^{2\alpha}$.
Using Proposition~\ref{prop:reverse-holder-ineq}, we bound this quantity in terms of the $L^1$-norm,
\begin{align*}
\left\|A\right\|_F^2 \int_{S^{n-1}}  \left\|[x]_d [x]_d^{\sf T}\right\|_F^2 \, \D \sigma(x)
&\leq \left\|A\right\|_F^2 \, 4^{4d} \, \biggl( \, \int_{S^{n-1}}  \biggl| \sum_{\left|\alpha\right|=d} x^{2\alpha} \biggr| \, \D \sigma(x) \, \biggr)^2\\
&= \left\|A\right\|_F^2 \, 4^{4d} \, \biggl( \, \int_{S^{n-1}}  \sum_{\left|\alpha\right|=d} x^{2\alpha} \, \D \sigma(x) \, \biggr)^2\\
&= \left\|A\right\|_F^2 \, 4^{4d} \, \biggl( \, \sum_{\left|\alpha\right|=d} \, \int_{S^{n-1}}  x^{2\alpha} \, \D \sigma(x) \, \biggr)^2,
\end{align*}
where each integral appearing in the sum is a diagonal entry of the moment matrix $M$. None of the diagonal entries of $M$ exceeds $\lambda_{\max}(M)$, and the claim follows:
\begin{equation*}
\left\|A\right\|_F^2 \, 4^{4d} \, \biggl( \, \sum_{\left|\alpha\right|=d} \, \lambda_{\max}(M) \, \biggr)^2
= \left\|A\right\|_F^2 \, 4^{4d} \, N(n,d)^2 \, \lambda_{\max}(M)^2.
\end{equation*}

To prove \eqref{ineq:Frobenius-leq-L2}, we apply Hölder's inequality,
\begin{equation*}
\left\|p\right\|_{L^2}^2 
= \int_{S^{n-1}} \langle A, [x]_d [x]_d^{\sf T} \rangle^2 \, \D \sigma(x)
\geq \biggl( \int_{S^{n-1}} \left| \langle A, [x]_d [x]_d^{\sf T} \rangle \right| \, \D \sigma(x) \biggr)^2.
\end{equation*}
The matrix $[x]_d [x]_d^{\sf T}$ is positive semidefinite for every $x$. Hence, $\langle A, [x]_d [x]_d^{\sf T} \rangle \geq 0$ because $A$ is positive semidefinite. Therefore,
$\left| \langle\, A, [x]_d [x]_d^{\sf T}\, \rangle \right| = \langle\, A, [x]_d [x]_d^{\sf T}\, \rangle$, and
\begin{equation}
\label{eq:integral-is-trace-with-moment}
\begin{split}
\int_{S^{n-1}} \langle\, A, [x]_d [x]_d^{\sf T} \rangle \, \D \sigma(x)
& = \int_{S^{n-1}} \sum_{\left|\alpha\right|,\left|\beta\right|=d} A_{\alpha,\beta} x^{\alpha+\beta} \, \D \sigma(x)\\
& = \sum_{\left|\alpha\right|,\left|\beta\right|=d} A_{\alpha,\beta} \, \int_{S^{n-1}} \, x^{\alpha+\beta}  \, \D \sigma(x)\\
& = \langle A, M \rangle.
\end{split}
\end{equation}
We have $\langle A , M \rangle \geq \lambda_{\min}(M) \Tr (A)$, and  finally,
\begin{equation*}
  \left\|p\right\|_{L^2}^2 \geq \lambda_{\min}(M)^2 \Tr (A)^2 \geq \lambda_{\min}(M)^2 \left\|A\right\|_F^2. \qedhere
\end{equation*}
\end{proof}

\begin{proof}[Proof of Theorem~\ref{th:projection-sos}]
Let $p$ be a rational, homogeneous polynomial of degree $2d$ normalized so that $\int_{S^{n-1}} p(x) \, \D \sigma(x) = 1$. The following semidefinite program computes a Gram matrix of a sum-of-squares polynomial closest to $p$:
\begin{equation}
\label{Gram-problem}
\begin{split}
\min \quad & \; t\\
\text{such that} \quad & \; \left\|((p - p_A)_\alpha ) \right\| \leq t,\\
& \; \int_{S^{n-1}} p_A(x) \, \D \sigma(x) = 1,\\
& \; p_A = [x]_d^{\sf T} A [x]_d,\\
& \; A \in \mathcal{S}^{N(n,d)}_{+}.
\end{split}
\end{equation}
We shall show that one can solve this semidefinite program in polynomial time, in the sense of~Theorem~\ref{th:sdppolynomialtime}.

We choose a Gram matrix $X_0$ of the polynomial $\omega^d = (x_1^2 + \cdots + x_n^2)^d$ to be the center of the balls $X_0+B(r)$ and $X_0+B(R)$. One can compute such a Gram matrix $X_0$ explicitly by using Hilbert's identity (see, for example, Reznick \cite[Section 8]{Reznick1992}), which states
\[
\int_{S^{n-1}}  (u^{\sf T} x)^{2d} \, \D \sigma(u) 
= B(n,2d) \, \omega^d,
\]
with $B(n,2d)= \prod\limits_{j=0}^{d-1} \dfrac{2j+1}{2j+n}$.
Therefore, 
\begin{equation*}
\omega^d
= \sum\limits_{\lvert \alpha \rvert, \lvert \beta \rvert = d}
x^\alpha \, \left( B(n,2d)^{-1} \, \frac{(d!)^2}{\alpha! \, \beta!} \, \int_{S^{n-1}} u^{\alpha+ \beta} \, \D \sigma(u) \right) \, x^\beta.
\end{equation*}
For multi-indices $\alpha, \beta$ with $\lvert \alpha \rvert,\lvert \beta \rvert =d$,
we define the matrix entry
\begin{equation*}
  [X_0]_{\alpha, \beta} = B(n,2d)^{-1} \, \frac{(d!)^2}{\alpha! \, \beta!} \, \int_{S^{n-1}} u^{\alpha+ \beta} \, \D \sigma(u),
\end{equation*}
so that $\omega^d = [x]_d^{\sf T} X_0 [x]_d$.

Again, suitable radii $r$ and $R$ are given by the Loewner-John ellipsoid computation of Blekherman. The radii of the corresponding balls in Proposition~\ref{prop:blekherman-inner-ball} and Proposition~\ref{prop:blekherman-outer-ball} are given in the $L^2$-norm on $\R[x]_{2d}$, but through Lemma~\ref{lem:relation-L2-Frobenius-norm} one can translate them efficiently to the Frobenius norm on $\mathcal{S}^{N(n,d)}$. Hence, Theorem~\ref{th:sdppolynomialtime} is applicable for the semidefinite program above.
\end{proof}

\section*{Acknowledgements}

V.M. and F.V. would like to thank the organizers of the Lorentz Center workshop \textit{Semidefinite Programming: Applications \& Solution Methods} (July 7--11, 2025)---David de Laat, Fernando M. de Oliveira Filho, and Angelika Wiegele---during which the collaboration that resulted into this paper was initiated. 
This work has been supported by European Union’s HORIZON–MSCA-2023-DN-JD programme under the Horizon Europe (HORIZON) Marie Skłodowska-Curie Actions, grant agreement 101120296 (TENORS).

\end{document}